\documentclass[12pt]{article}
\usepackage{amsfonts}
\usepackage[T1]{fontenc}
\usepackage[utf8]{inputenc}
\usepackage[letterpaper]{geometry}
\usepackage{amsmath, mathrsfs}
\usepackage{graphicx}
\usepackage[active]{srcltx}
\usepackage[textsize=small]{todonotes}
\numberwithin{equation}{section}

\usepackage{amsmath}
\usepackage{graphicx}
\usepackage[active]{srcltx}
\usepackage{amssymb}
\usepackage{color}
\usepackage{amsmath}
\usepackage{graphicx}
\usepackage[active]{srcltx}
\include{srctex}

\newtheorem{theorem}{Theorem}[section]

\newtheorem{definition}[theorem]{Definition}

\newtheorem{lemma}[theorem]{Lemma}

\newtheorem{proposition}[theorem]{Proposition}
\newtheorem{remark}[theorem]{Remark}

\newenvironment{proof}[1][Proof]{\textbf{#1.} }{\ \rule{0.5em}{0.5em}}

\def\E{\mathbb{E}}

\def\R{\mathbb{R}}
\def\F{\mathcal{F}}
\def\H{\mathcal{H}}

\def\wh{\widehat}

\def\wh{\widehat}
\def\v{\varphi}
\def\wt{\widetilde}

\begin{document}

\title{ Fractional stochastic wave equation driven by a Gaussian noise rough in space}
\author{Jian Song, Xiaoming Song, and Fangjun Xu}
\date{}
\maketitle
\begin{abstract}
In this article, we consider fractional stochastic wave equations on $\mathbb R$ driven by a multiplicative Gaussian noise which is white/colored in time and has the covariance of a fractional Brownian motion with Hurst parameter $H\in(\frac14, \frac12)$ in space. We prove the existence and uniqueness of the mild Skorohod solution, establish lower and upper bounds for the $p$-th moment of the solution for all $p\ge2$, and obtain the H\"older continuity in time and space variables for the solution. 
\end{abstract}

\vspace{1mm}


\section{Introduction}
Consider the following fractional stochastic wave equation (SWE) on $\R$
\begin{equation}\label{swe}
\begin{cases}
 \dfrac{\partial^2 u}{\partial t^2}(t,x)  =-(-\Delta)^{
 \frac\kappa2} u(t,x)+u\dot{W}(t,x), \quad t>0\\
 u(0,x)  =  1,  \frac{\partial u}{\partial t}(0,x)  =  0,
 \end{cases}
\end{equation}
with $\kappa\in(0, 2]$, where $\dot W(t,x)$ is a Gaussian noise with covariance 
\[\E[\dot W(t,x)\dot W(s,y)]=f_0(t-s)f(x-y).\]
In this article, we assume that the noise  is rough in space, i.e., $f(x)=(|x|^{2H})''$ with $H\in(\frac14,\frac12)$, where $(|x|^{2H})''$ means the second derivative of $|x|^{2H}$ in the sense of distribution. Note that for fractional Brownian motion $B^H$ with Hurst parameter $H\in(0,1)$, its derivative (in the sense of distribution) $\dot B^H$ has the covariance $\E[\dot B^H(t) \dot B^H(s)]=H(2H-1)f(t-s)$.  We also assume that  the temporal covariance function $f_0(t)$ is either the Dirac delta function $\delta(t)$ or a nonnegative and nonnegative-definite function such that  $f_0(t)\sim |t|^{2H_0-2}$ with  $H_0\in (\frac12,1)$, i.e.,  $c|t|^{2H_0-2}\le f_0(t)\le C|t|^{2H_0-2}$ for some constants $0<c<C<\infty$.

 The It\^o-type probabilistic approach for stochastic partial differential equations (SPDEs) was established in \cite{walsh}, where Walsh introduced  martingale measures and defined stochastic integrals with respect to the martingale measures, and then SPDEs driven by space-time white noise were investigated.    Following Walsh's approach, SWEs on $\R^d$ with $d\le 2$ were studied, for instance in \cite{df, ms, mueller}.  In \cite{dalang99}, Dalang extended Walsh's stochastic integral and applied it to solve SPDEs whose Green's function is not a function but a Schwartz distribution. In particular Dalang's theory is applicable to  SWEs in $d$-dimension with $d\ge 3$,  and we refer to \cite{cd08, ds09, hhn14} and the references therein for the study of SWEs in high dimensions.   
 
For SPDEs driven by a multiplicative Gaussian noise which is colored in time (i.e, the temporal covariance $f_0$ is not the Dirac delta function), the probabilisitic approach based on martingale properties cannot be applied directly since the noise does not have martingale structure in time. An alternative approach is to apply  Malliavin calculus to study the chaos expansion of the Skorohod solution, see, for instance,  \cite{hhns15, hn09, hns11, s17} for stochastic heat equations (SHEs) and \cite{balan12, bs17} for SWEs. 

In this article, we shall prove the existence and uniqueness of the mild Skorohod solution to \eqref{swe} (Theorem \ref{thm1}), establish lower and upper bounds for the $p$-th moment of the solution for all $p\ge2$ (Proposition \ref{prop1}), and obtain the H\"older continuity for the solution in time and space variables (Proposition \ref{prop2}). In the following, we briefly describe some related recent development on SHEs and SWEs driven by multiplicative Gaussian noise.

Hu and Nualart \cite{hn09} investigated SHEs driven by a multiplicative fractional Brownian sheet that is colored in time and white in space. Hu et al.  \cite{hns11} obtained Feynman-Kac formulae for solutions of SHEs driven by a fractional Brownian sheet, and used them to investigate the regularity of the solutions. The result in \cite{hns11} then was extended to SHEs driven by a general Gaussian noise in Hu et al. \cite{hhns15} and to SHEs with the Laplacian operator being replaced  by the infinitesimal generator of a symmetric L\'evy process in Song \cite{s17}. The noise considered in the papers \cite{hhns15, hn09, hns11, s17} is not ``rough'', as its spatial covariance corresponds to that of fractional Brownian motion with Hurst parameter $H\ge \frac12$. The SHEs and SWEs on $\R$ driven by a Gaussian noise that is white in time and rough in space were investigated in Hu et al. \cite{hhlnt}  and in Balan et al. \cite{bjq15}, respectively.  Recently, Chen conducted a systematic investigation on  SHEs with noise that is rough in space and/or in time  in \cite{chen18, chen19}.

If the solution of a dynamic system with noise develops very high peaks, it is said that the system possesses the intermittency property.  The concept of intermittency arose in physics, and in mathematics it is related to the long-term asymptotics of the moments of the solution.  The intermittency property was studied, for instance, in \cite{bc16, bc95, cd15, ck12, fk13, hhns15, s12} for heat equations, and in \cite{bc16,cjks, dm09} for wave equations. In particular,  precise long-term asymptotics for SHEs was obtained in \cite{chen17, chss, chsx, hln, hln17}, and the second order  Lyapunov exponent  for SWEs was obtained in \cite{bs19}.

For the H\"older continuity of SHEs driven by multiplicative Gaussian noise colored in time, we refer to \cite{bqs, hhns15, hl, s17} and the references therein.  For SWEs with noise white in time,  H\"older continuity of the solutions was studied in  \cite{df} for the spatial dimension $d=2$,  in \cite{ds09, hhn14} for $ d=3$, and in \cite{cd08} for general dimensions; for SWEs with noise correlated in time and space,  H\"older continuity was established in \cite{balan12, bs17} for general dimensions.

Finally, we would like to make some comments on our results. 
\begin{itemize}
\item[(a)]  Note that we require $H>\frac14$ for SWEs in Theorem \ref{thm1}, which was also assumed in \cite{hln} for SHEs with rough spatial noise.  Nevertheless, the approach used in the proof of Theorem \ref{thm1} can be also applied to SHEs  and relax the condition $H>\frac14$ to $H_0+H>\frac34$ (see Remark \ref{remark1}). 
 
\item[(b)] The rate of the bounds for the $p$-th moments obtained in Proposition \ref{prop1} is consistent with the known results in, for instance, \cite{bc16,bs19}.  The lower bound is relatively more difficult to establish. One of the obstacles is that the Fourier transform of the Green's function of the fractional wave equation is not a nonnegative function, and this issue is resolved by showing that the integral of the Fourier transform of the Green's function is positive (see Lemma \ref{lem6}).  


\item[(c)] The H\"older continuity obtained in Proposition \ref{prop2} is consistent with the known results (e.g., \cite[Theorem 5.1]{balan12}, \cite[Proposition 8.3]{bs17}, and \cite[Theorem 7.6]{cd08}) which dealt with SWEs driven by the noise that is not rough in space. The major difference/difficulty of obtaining the H\"older continuity for SWEs with rough spatial noise is the following. Denoting the spectral measure of the spatial covariance $\mu(d\xi)=\wh f(\xi) d\xi$,  the condition 
\begin{equation}\label{eq1.2}
\sup_{\eta\in\R}\int_{\R}\frac{1}{1+|\xi-\eta|^2}\mu(d\xi)<\infty
\end{equation}
plays a critical role in obtaining the H\"older continuity of the solution when the spatial covariance $f$ is a nonnegative, nonnegative definite, and locally integrable function (see, e.g., \cite{balan12, bs17, cd08}). However,  when $H<\frac12$, the spatial covariance $(|x|^{2H})''$ is a genuine distribution (see, e.g., \cite{jolis}), and condition \eqref{eq1.2} is indeed violated (see  \cite[Lemma A.1]{bjq15}). 
\end{itemize}

This article is organized as follows. In Section \ref{sec-preliminary}, some  preliminary results on Malliavin calculus associated with the noise $W$ are provided. In Section \ref{sec-solution}, the existence and uniqueness of the solution to \eqref{swe} is obtained under proper conditions. 
In Section \ref{sec-moment}, we derive the lower and upper bounds for the $p$-th moment of the solution for $p\ge2$ and then deduce the weak intermittency.  In Section \ref{sec-continuity}, the H\"older continuity of the solution in time and space is obtained. Finally, some lemmas used in the preceding sections are gathered in Section \ref{sec-appendix}.

\section{Preliminaries}\label{sec-preliminary}
In this section, we recall some preliminaries on Malliavin calculus associated with the
Gaussian noise $\dot W$. We refer to \cite{nualart} for more details.

Let $\H$ be the completion of the Schwartz space $\mathcal S(\R_+\times \R)$ under the inner product 
\begin{align}\label{eq-cov}
 \langle \varphi, \phi \rangle_\H=C_H\int_{\R_+^2}\int_{\R} f_0(r-s) \wh \varphi(r,\xi)\overline{\wh\phi(s,\xi)} \mu(d\xi) drds,
 \end{align}
 where 
 \begin{equation} \label{eq-CH}
 C_H=\frac{\Gamma(2H+1)\sin(\pi H)}{2\pi}
 \end{equation}
  and $\mu(d\xi)=|\xi|^{1-2H} d\xi$ with $H\in(0, \frac12)$.  Here, $\wh \varphi$ is the Fourier transform of $\varphi$ in the space, i.e.,  for $\varphi\in \mathcal S(\R_+\times \R)$, 
 \[\wh\varphi(s,\xi)=\int_\R e^{-i\xi x} \varphi(s,x) dx.\]  
 
In particular, if $\varphi$ is a measurable function such that $\wh \varphi$ is also a measurable function and
\[\int_{\R_+^2}\int_{\R} f_0(r-s) |\wh \varphi(r,\xi)| |\wh \varphi(s,\xi)| \mu(d\xi) drds<\infty,\]
then $\varphi\in\H.$ Note that $\H$ may contain distributions rather than just measurable functions if $f_0(r-s)\sim |r-s|^{2H_0-2}$ for some $H_0\in (\frac12, 1)$  (see \cite{jolis, pt}).

In a complete probability space $(\Omega, \F, P)$, let $W=\{W(\varphi), \varphi\in\H\}$ be an isnormal Gaussian process with the covariance 
 \[\E[W(\varphi)  W(\phi)]=\langle \varphi, \phi\rangle_\H,\]
 and we also denote
 \[W(\varphi)=\int_{\R_+}\int_\R \varphi(t,x) W(dt,dx).\]
We also call $W(\varphi)$  the Wiener integral of $\varphi$ with respect to $W$. In light of \cite[Theorem 3.1]{pt} and \cite[Proposition 4.1]{jolis}, the Gaussian family $\{W(\varphi), \varphi\in\H\}$ coincides with the linear expansion of the Gaussian family $\{W(t,x), (t,x)\in\R_+\times \R)\}$ with the covariance
\[\E[W(t,x)W(s,y)]=\frac12 \Big(|x|^{2H}+|y|^{2H}-|x-y|^{2H}\Big) \int_0^t\int_0^s f_0(r_1-r_2) dr_1dr_2,\]
and in particular $W(t,x)=W(I_{[0,t]\times[0,x]})$ with the convention $I_{[0,t]\times[x,0]}=-I_{[0,t]\times[0,x]}$ for $x<0.$

For the smooth and cylindrical random variables of the form $F=h(W(\varphi_1), \dots, W(\v_n))$ with $h$ being smooth and its partial derivatives having at most polynomial growth, the Malliavin derivative $DF$ of $F$ is the $\H$-valued random variable defined by 
\[ DF=\sum_{k=1}^n\frac{\partial h}{\partial x_k}(W(\v_1),\dots, W(\v_n)) \v_k. \]
Noting that $D$ is closable from $L^2(\Omega)$ to $L^2(\Omega; \H)$, we define the Sobolev space $\mathbb D^{1,2}$ as the closure of the space of the smooth and cylindrical random variables under the norm
\[\|D\|_{1,2}=\left(\E[F^2]+\E[\|DF\|_\H^2]\right)^{\frac12}. \]

The divergence operator $\delta$, which is also known as the Skorohod integral, is the adjoint of the Malliavin derivative operator $D$ defined by the duality  
\[\E[F\delta(u)]=E[\langle DF, u\rangle_\H],~ \forall F\in \mathbb D^{1,2}, ~\forall u\in \text{Dom } \delta.\]
 Here $\text{Dom } \delta$ is the domain of the divergence operator $\delta$, which is the space of the $\H$-valued random variables $u\in L^2(\Omega; \H)$ such that  $|E[\langle DF, u\rangle_\H]|\le c_F \|F\|_2$ with some constant $c_F$ depending on $F$, for all $F\in \mathbb D^{1,2}$. Thus, for $u\in \text{Dom } \delta$, $\delta(u)\in L^2(\Omega)$. In particular, $\E[\delta(u)]=0.$ We also use the following notation
 \[\delta(u)=\int_{\R_+}\int_\R u(t,x) W(dt,dx), ~~u\in \text{Dom } \delta.\]

Now we recall the Wiener chaos expansion. Let $\mathbb H_0=\R$, and for any integer $n\ge1$, let $\mathbb H_n$ be the closed linear subspace of $L^2(\Omega)$ containing the set of random variables $\{H_n(W(\v)), \v\in\H, \|\v\|_{\H}=1\}$, where $H_n$ is the $n$-th Hermite polynomial, i.e., $H_n(x)=(-1)^ne^{x^2}\frac{d^n}{dx^n}(e^{-x^2})$.  Then $\mathbb H_n$ is called the $n$-th Wiener chaos of $W$.  Denoting by $\mathcal F$ the $\sigma$-field generated by $\{W(\varphi), \varphi\in \H\}$, then we have the following Wiener chaos decomposition
 \[
  L^2(\Omega, \mathcal F, P)=\oplus_{n=0}^\infty \mathbb H_n.
  \] 
For $n\ge1$, denote by $\H^{\otimes n}$ the $n$-th tensor product of $\H$, and let $\wt \H^{\otimes n}$ be the symmetrization of $\H^{\otimes n}$. Then the mapping $I_n(h^{\otimes n})=H_n(W(h))$ for any $h\in \H$ can be extended to a linear isometry between $\wt\H^{\otimes n}$ and the $n$-th Wiener chaos $\mathbb H_n$. Thus, for any random variable $F\in L^2(\Omega, \mathcal F, P)$,  it has the following unique Wiener chaos expansion in the sense of $L^2(\Omega)$, 
 \[F=\E[F]+\sum_{n=1}^\infty I_n(f_n) \text{ with } f_n\in \wt\H^{\otimes n}.\]

Throughout the paper, the generic constant $C$  varies at different places.

\section{Existence and uniqueness of the solution}\label{sec-solution}

 In this section, we  obtain the existence and uniqueness of the mild Skorohod solution to \eqref{swe} under some conditions in Theorem \ref{thm1}, and we show that in Proposition \ref{prop3.4} these conditions are also necessary if the noise is white in time.
 
  Let $G_t(x)$ be the fundamental solution of the equation $\frac{\partial^2}{\partial t^2} u+(-\Delta)^{\frac\kappa 2} u=0 $ on $\R^d$, then its Fourier transform in space $\wh G_t(\xi)$ solves the following equation
 \[
 \frac{\partial^2\wh G_t(\xi)}{\partial t^2}+|\xi|^\kappa \wh G_t(\xi)=0,
 \]
  and it is given by (see \cite[Section 2.2]{fkn}; \cite[Example 6]{dalang99} and \cite[Chapter 1 Section 7]{treves} for the case $\kappa=2$)
  \[\wh G_t(\xi)=\frac{\sin(t|\xi|^{\kappa/2})}{|\xi|^{\kappa/2}}.\] Recall that when 
  $\kappa=2$, the Green's function $G_t(x)$ is a measurable function for $d\le 2$,
\begin{equation*}
G_t(x)=\begin{cases} \dfrac{1}{2}I_{\{|x| <t\}}, & \mbox{if} \ d=1,\\
\quad \\
\dfrac{1}{2\pi}\dfrac{1}{\sqrt{t^2-|x|^2}}I_{\{|x|<t\}} ,& \mbox{if} \ d=2,
\end{cases}
\end{equation*}  
$G_t(\cdot)=\frac{1}{4\pi t}\sigma_t$ for  $d=3$, where $\sigma_t$ is the surface measure on the sphere $\{x \in \R^3; |x|=t\}$, and $G_t(\cdot)$ is a genuine distribution with compact support in $\R^d$ if $d\ge3$.  Note that when $\kappa\in(1,2)$ and $d=1$, the Green's function $G_t(x)\in L^2(\R)$ for all $t\ge 0$ as its Fourier transform  $\frac{\sin(t|\xi|^{\kappa/2})}{|\xi|^{\kappa/2}}\in L^2(\R).$

 We consider the following filtration
\[
\mathcal F_t=\sigma\{W(I_{[0,s]}\varphi), 0\le s\leq t, \varphi\in \mathcal S(\mathbb{R})\}\vee \mathcal N,
\]
where $\mathcal N$ denotes the collection of null sets.
\begin{definition}
An adapted random field $u=\{u(t,x), t\ge 0, x\in\R\}$ is a mild 
Skorohod solution to
\eqref{swe} if $\E[u^2(t,x)]<\infty$ for all $(t,x)\in \R_+\times\R$ and it satisfies the following integral equation
\begin{equation}\label{spde-skr}
u(t,x)=1+\int_0^t\int_{\R} G_{t-s}(x-y) u(s,y) W(ds,dy), 
\end{equation}
where the integral on the right-hand side is a Skorohod integral. 
\end{definition}

Note that if $\E[|u(t,x)|^2]<\infty$, the solution has a unique Wiener chaos expansion 
\[
u(t,x)=\sum_{n=0}^\infty I_n(g_n(\cdot, t, x))
\]
with $g_n(\cdot, t,x)\in \wt \H^{\otimes n}$.   Now assume that $u(t,x)$ is a mild Skorohod solution to \eqref{swe}. Let $\mathcal P_n$ be the set of permutations on $\{1,2, \dots, n\}$.  Following the approach used in \cite[Section 4.1]{hn09}, we get 
\begin{equation}\label{eq-hn}
g_n(s_1,\dots, s_n, x_1,\dots, x_n, t,x)=\frac{1}{n!} G_{t-s_{\rho(n)}}(x-x_{\rho(n)})\cdots G_{s_{\rho(2)}-s_{\rho(1)}}(x_{\rho(2)}-x_{\rho(1)}),
\end{equation}
where $\rho\in \mathcal P_n$ is the permutation such that $0<s_{\rho(1)}<s_{\rho(2)}<\cdots<s_{\rho(n)}<t.$ Thus, to prove the existence and uniqueness of the solution to \eqref{swe} is equivalent to prove 
\begin{equation}
\E[|u(t,x)|^2]=\sum_{n=0}^\infty n! \|g_n(\cdot, t,x)\|^2_{\H^{\otimes n}}<\infty.\label{2mom}
\end{equation}

\begin{theorem}\label{thm1}
Assume that $H_0\in[\frac12,1)$, $H\in(\frac14,\frac12)$ and $\kappa\in(3-4H,2]$.  Then there exists a unique square integrable mild Skorohod solution to \eqref{swe}.
\end{theorem}
\begin{proof} It suffices to prove \eqref{2mom}. 
 We use the notation ${\boldsymbol \xi}:=(\xi_1,\dots, \xi_n)$ and similarly  for $\bf{s} $, $\bf{r}$ and $\boldsymbol {\mu}(d\boldsymbol{\xi})$. 

We first consider the case $H_0\in(\frac12,1)$. Since we assume that $f_0(s)\sim |s|^{2H_0-2}$  for $H_0\in (\frac12, 1)$, throughout the rest of the article, we will simply assume $f_0(s)=|s|^{2H_0-2}$ in this case. Note that
\begin{align}
  &n! \|g_n(\cdot, t,x)\|^2_{\H^{\otimes n}}\notag\\
= &  n! \int_{\R^{n}} \int_{[0,t]^{2n}} \F g_n({\bf s},
  \cdot,t,x)(\boldsymbol{\xi})\overline{\F g_n({\bf r},\cdot, t,x)(\boldsymbol{\xi})} \prod_{j=1}^n |s_j-r_j|^{2H_0-2}d{\bf s}d{\bf r}  \boldsymbol {\mu} (d\boldsymbol{\xi})\label{chaos}
    \end{align}
  with
   \begin{equation*}
   \F g_n({\bf s},  \cdot,t,x)(\boldsymbol{\xi})=\frac1{n!} e^{-ix (\xi_1+\dots+\xi_n)}\prod_{j=1}^n \frac{\sin((s_{\rho(j+1)}-s_{\rho(j)})|\xi_{\rho(1)}+\dots+\xi_{\rho(j)}|^{\kappa/2})}{|\xi_{\rho(1)}+\dots+\xi_{\rho(j)}|^{\kappa/2}},
  \end{equation*}
  where we  use the convention $s_{\rho(n+1)}=t.$  Thus, by Lemma B.3 in \cite{bc16} (see also \cite{mmv}) and a change of variables, we have
  \begin{align}
 & n! \|g_n(\cdot, t,x)\|^2_{\H^{\otimes n}}\notag\\
 \le & n! \bigg(\int_{[0,t]^{n}}\Big(\int_{\R^{n}} |\F g_n({\bf s},\cdot, t,x)(\boldsymbol\xi)|^2 \boldsymbol {\mu}(d\boldsymbol{\xi})\Big)^{\frac1{2H_0}}d{\bf s}\bigg)^{2H_0}\notag\notag\\
 =& \frac1{n!} \bigg(\int_{[0,t]^{n}}\Big(\int_{\R^{n}}  \prod_{j=1}^n \frac{\sin^2((s_{\rho(j+1)}-s_{\rho(j)})|\xi_{\rho(1)}+\dots+\xi_{\rho(j)}|^{\kappa/2})}{|\xi_{\rho(1)}+\dots+\xi_{\rho(j)}|^{\kappa}} \boldsymbol {\mu}(d\boldsymbol{\xi})\Big)^{\frac1{2H_0}}d{\bf s}\bigg)^{2H_0}\notag\\
 =& (n!)^{2H_0-1} \bigg(\int_{[0,t]_<^{n}}\Big(\int_{\R^{n}} \prod_{j=1}^n \frac{\sin^2((s_{j+1}-s_{j})|\xi_{1}+\dots+\xi_{j}|^{\kappa/2})}{|\xi_{1}+\dots+\xi_{j}|^{\kappa}}
 \boldsymbol {\mu}(d\boldsymbol{\xi})\Big)^{\frac1{2H_0}}d{\bf s}\bigg)^{2H_0}\notag\\
 =& (n!)^{2H_0-1} \bigg(\int_{[0,t]_<^{n}}\Big(\int_{\R^{n}} \prod_{j=1}^n \frac{\sin^2((s_{j+1}-s_{j})|\eta_{j}|^{\kappa/2})}{|\eta_{j}|^{\kappa}} |\eta_j-\eta_{j-1}|^{1-2H}d\boldsymbol{\eta}\Big)^{\frac1{2H_0}}d{\bf s}\bigg)^{2H_0}, \label{eq2.4}
    \end{align}
 where  $ [0,t]^n_<=[0=s_0<s_1<\dots<s_n<s_{n+1}=t]$.

Let $\mathcal A_n$ be a subset of the index set  $\big\{(\alpha_1, \dots, \alpha_n)\in \{0, 1,2\}^n\big\}$ such that 
\[x_1\prod_{j=2}^n(x_j+x_{j-1})=\sum_{\alpha\in \mathcal A_n}\prod_{j=1}^n x_j^{\alpha_j}.\]
Then $\#(\mathcal A_n)=2^{n-1}$, and for each $\alpha\in\mathcal A_n$,  we have the following properties: $\alpha_1\in\{1,2\}, \alpha_n\in \{0, 1\},~ \alpha_2, \dots, \alpha_{n-1}\in \{0, 1, 2\}, ~\sum_{j=1}^n \alpha_j=n$ and $\alpha_j+\alpha_{j+1}\in\{1,2,3\}$ for $j=1,2,\dots, n-1$.  Hence, noting that $|a+b|^{1-2H}\le |a|^{1-2H}+|b|^{1-2H}$, we get
\begin{equation}
\prod_{j=1}^n |\eta_j-\eta_{j-1}|^{1-2H}= |\eta_1|^{1-2H}\prod_{j=2}^n |\eta_j-\eta_{j-1}|^{1-2H}\le  \sum_{\alpha\in \mathcal A_n}\prod_{j=1}^n |\eta_j|^{(1-2H)\alpha_j}.\label{eq-prod}
\end{equation}
 Using \eqref{eq-prod} and the fact that $(\sum x_m)^{\frac{1}{2H_0}}\leq \sum x^{\frac{1}{2H_0}}_m$ for all $x_m\geq 0$, the estimation \eqref{eq2.4} now becomes
  \begin{align}
 & n! \|g_n(\cdot, t,x)\|^2_{\H^{\otimes n}}\notag\\
 \le & (n!)^{2H_0-1} \bigg(\int_{[0,t]_<^{n}}\Big(\int_{\R^{n}} \prod_{j=1}^n \frac{\sin^2((s_{j+1}-s_{j})|\eta_{j}|^{\kappa/2})}{|\eta_j|^\kappa} |\eta_j-\eta_{j-1}|^{1-2H}d\boldsymbol{\eta}\Big)^{\frac1{2H_0}}d{\bf s}\bigg)^{2H_0}\notag\\
 \le & (n!)^{2H_0-1} \bigg(\int_{[0,t]_<^{n}}\Big(\sum_{\alpha\in\mathcal A_n}\int_{\R^{n}} \prod_{j=1}^n \frac{\sin^2((s_{j+1}-s_{j})|\eta_{j}|^{\kappa/2})}{|\eta_j|^\kappa} |\eta_j|^{\alpha_j(1-2H)}d\boldsymbol{\eta}\Big)^{\frac1{2H_0}}d{\bf s}\bigg)^{2H_0}\notag\\
 \le & (n!)^{2H_0-1} \bigg(\int_{[0,t]_<^{n}}\sum_{\alpha\in\mathcal A_n}\Big(\int_{\R^{n}} \prod_{j=1}^n \frac{\sin^2((s_{j+1}-s_{j})|\eta_{j}|^{\kappa/2})}{|\eta_j|^\kappa} |\eta_j|^{\alpha_j(1-2H)}d\boldsymbol{\eta}\Big)^{\frac1{2H_0}}d{\bf s}\bigg)^{2H_0}\notag\\ 
= & (n!)^{2H_0-1} \Bigg(\int_{[0,t]_<^{n}}\sum_{\alpha\in\mathcal A_n} \Big(\frac2\kappa\Big)^{\frac{n}{2H_0}} \prod_{j=1}^n  (s_{j+1}-s_j)^{\frac1{2H_0}[2-\frac2\kappa-\frac2\kappa   \alpha_j(1-2H)]}\notag \\
&\qquad\qquad \qquad  \left(\int_{\R}\frac{\sin^2(\eta)}{\eta^2} |\eta|^{\frac2\kappa\alpha_j(1-2H)+\frac2\kappa-1}d\eta\right)^{\frac1{2H_0}}d{\bf s}\Bigg)^{2H_0}.\label{eq-2.6} 
    \end{align}
It follows from  Lemma \ref{lem5} and the condition $\kappa>3-4H$ that for all $\alpha_j\in\{0, 1, 2\}$ 
\begin{align*}
&\int_{\R}\frac{\sin^2(\eta)}{\eta^2} |\eta|^{\frac2\kappa\alpha_j(1-2H)+\frac2\kappa-1}d\eta<\infty.
\end{align*}   
 Lemma \ref{lem5} is applicable here, since the condition $\kappa>3-4H$ implies $\frac2\kappa\alpha_j(1-2H)+\frac2\kappa-3\in (-3,-1)$ for all $\alpha_j\in\{0,1, 2\}$.
   
    Therefore, one can find a positive constant $C$ depending only on $(\kappa, H_0, H)$ such that
    \begin{align*}
    & n! \|g_n(\cdot, t,x)\|^2_{\H^{\otimes n}}\\
    \le & C^n (n!)^{2H_0-1} \bigg(\int_{[0,t]_<^{n}}\sum_{\alpha\in \mathcal A_n} \prod_{j=1}^n  (s_{j+1}-s_j)^{\frac1{2H_0}[2-\frac2\kappa-\frac2\kappa   \alpha_j(1-2H)]}  d{\bf s}\bigg)^{2H_0}.
    \end{align*}
 For each fixed $\alpha\in \mathcal A_n$, denote 
$
\beta_j=\frac1{2H_0}\left[2-\frac2\kappa-\frac2\kappa   \alpha_j(1-2H)\right],\  j=1, \dots, n,
$ and 
$
\beta=\sum_{j=1}^n \beta_j=\frac{n}{H_0}\left[\left(1-\frac2\kappa\right)+\frac{2H}{\kappa} \right].
$ Note that  $H_0\in(\frac12,1)$, $H\in(\frac14,\frac12)$ and $\kappa>3-4H$ implies that  $\beta_j>0 $ and $\beta>0$.  By Lemma \ref{lem2} we have 
\[\int_{[0,t]_<^{n}}  \prod_{j=1}^n  (s_{j+1}-s_j)^{\beta_j}   d{\bf s}=\frac{\prod_{i=1}^n\Gamma(1+\beta_i)t^{n+\beta}}{\Gamma(n+1+\beta)}. \]
 Therefore, from \eqref{gamma-appr} in Lemma \ref{lem22} with $a=1+\frac1{H_0}\left[\left(1-\frac{2}{\kappa}\right)+\frac{2H}{\kappa}\right]$ and $b=1$, and the fact  $\#(\mathcal A_n)=2^{n-1}$, it follows that there exists some positive constant $C$ such that,
\begin{align}
& n! \|g_n(\cdot, t,x)\|^2_{\H^{\otimes n}}\notag\\
\le &\ C^n (n!)^{2H_0-1} \bigg(\frac{t^{n+\beta}}{\Gamma(n+1+\beta)}\bigg)^{2H_0}\notag\\
=&\ \frac{C^n (n!)^{2H_0-1}t^{n\left(2H_0+2\left[\left(1-\frac2\kappa\right)+\frac{2H}\kappa\right]\right)}}{\Gamma(an+1)^{2H_0}}
\sim \frac{ C^nt^{n\left(2H_0+2\left[\left(1-\frac2\kappa\right)+\frac{2H}\kappa\right]\right)}}{(n!)^{2\left[\left(1-\frac{2}{\kappa}\right)+\frac{2H}{\kappa}\right]+1} a^{2H_0an+H_0}n^{H_0(1-a)}}\label{eq-2.7}
\end{align}
when $n\to \infty$. Notice that there exists $
\lambda>1$ such that $\lambda^{-n}\leq a^{2H_0an+H_0} n^{H_0(1-a)} \leq \lambda^n$ for all $n$. Hence,  by \eqref{power-est} in Lemma \ref{lem22},  there exists a positive constant $C$ such that 
\begin{align}\label{eq2.7}
 \E[|u(t,x)|^2]=\sum_{n=0}^\infty n! \|g_n(\cdot, t,x)\|^2_{\H^{\otimes n}}&\le \sum_{n=0}^\infty \frac{C^nt^{n\left(2H_0+2\left[\left(1-\frac2\kappa\right)+\frac{2H}\kappa\right]\right)}}{(n!)^{2\left[\left(1-\frac{2}{\kappa}\right)+\frac{2H}{\kappa}\right]+1}}\notag\\
 &\le C \exp\left(Ct^{\frac{2\kappa H_0+2(\kappa-2)+4H}{3\kappa-4+4H}}\right).
 \end{align}

Next, we consider the case $H_0=\frac12$, i.e, $f_0(t)=\delta(t)$, and we have 
\begin{align*}
  &n! \|g_n(\cdot, t,x)\|^2_{\H^{\otimes n}}\notag\\
= &  n! \int_{\R^{n}} \int_{[0,t]^{n}} \left|\F g_n({\bf s},  \cdot,t,x)(\boldsymbol{\xi})\right|^2d{\bf s}  \boldsymbol {\mu} (d\boldsymbol{\xi})\\
  =& \frac1{n!} \int_{[0,t]^{n}}\int_{\R^{n}}  \prod_{j=1}^n \frac{\sin^2((s_{\rho(j+1)}-s_{\rho(j)})|\xi_{\rho(1)}+\dots+\xi_{\rho(j)}|^{\kappa/2})}{|\xi_{\rho(1)}+\dots+\xi_{\rho(j)}|^{\kappa}} \prod_{j=1}^n|\xi|^{1-2H}d{\boldsymbol \xi}d{\bf s}\\
 =& \int_{[0,t]_<^{n}}\int_{\R^{n}} \prod_{j=1}^n \frac{\sin^2((s_{j+1}-s_{j})|\xi_{1}+\dots+\xi_{j}|^{\kappa/2})}{|\xi_{1}+\dots+\xi_{j}|^{\kappa}}
\prod_{j=1}^n|\xi|^{1-2H}d{\boldsymbol \xi}d{\bf s}\\
 =& \int_{[0,t]_<^{n}}\int_{\R^{n}} \prod_{j=1}^n \frac{\sin^2((s_{j+1}-s_{j})|\eta_{j}|^{\kappa/2})}{|\eta_{j}|^{\kappa}} \prod_{j=1}^n|\eta_j-\eta_{j-1}|^{1-2H}d\boldsymbol{\eta}d{\bf s}.
     \end{align*}
 The last term in the above equation equals the right-hand side of   \eqref{eq2.4} with $H_0=\frac12$. Analogue to the arguments in \eqref{eq-prod}-\eqref{eq2.7}, we shall get the following estimation for the second moment 
 \begin{equation}\label{eq2.8}
 \E[|u(t,x)|^2]=\sum_{n=0}^\infty n! \|g_n(\cdot, t,x)\|^2_{\H^{\otimes n}}\le C\exp \left(Ct\right).
 \end{equation}
We complete the proof.\hfill
\end{proof}

\begin{remark}\label{remark1}
For SHEs on $\R$ driven by a multiplicative Gaussian noise that is rough in space, the existence and uniqueness of the mild Skorohod solution  was obtained in \cite{hhlnt} for the noise white in time and in \cite{hln} for the noise colored in time. The condition $H>\frac14$ was assumed in both \cite{hhlnt} and \cite{hln}.  However, the method used in the proof of the above theorem suggests that the condition can be reduced to $H_0+H>\frac34,$ and this is consistent with the result in \cite{chen18}.  Indeed, for the following SHE on $\R$,
\begin{equation*}
\begin{cases}\dfrac{\partial u^h}{\partial t}(t,x)=\frac12 \Delta u^h(t,x)+u^h\dot W(t,x), ~~ t>0\\
u^h(0,x)=1, 
\end{cases}
\end{equation*}
 the Green's function is the heat kernel $G_t^h(x)=\frac1{\sqrt{2\pi t}}e^{-\frac{|x|^2}{2t}}$. Consequently, the Wiener chaos expansion of the solution is 
 \[u^h(t,x)=\sum_{n=0}^\infty I_n(g_n^h(\cdot, t, x)),\]
where
\begin{equation*}
g_n^h(s_1,\dots, s_n, x_1,\dots, x_n, t,x)=\frac{1}{n!} G^h_{t-s_{\rho(n)}}(x-x_{\rho(n)})\cdots G^h_{s_{\rho(2)}-s_{\rho(1)}}(x_{\rho(2)}-x_{\rho(1)})
\end{equation*}
and 
 \begin{equation*}
   \F g^h_n({\bf s},  \cdot,t,x)(\boldsymbol{\xi})=\frac1{n!} e^{-ix (\xi_1+\dots+\xi_n)}\prod_{j=1}^n \exp\left[- \frac12(s_{\rho(j+1)}-s_{\rho(j)})|\xi_{\rho(1)}+\dots+\xi_{\rho(j)}|^2\right].
  \end{equation*}
Now, the second moment of each chaos is
\begin{align*}
  &n! \|g^h_n(\cdot, t,x)\|^2_{\H^{\otimes n}}\\
= &  n! \int_{\R^{n}} \int_{[0,t]^{2n}} \F g^h_n({\bf s}
  \cdot,t,x)(\boldsymbol{\xi})\overline{\F g^h_n({\bf r},\cdot, t,x)(\boldsymbol{\xi})} \prod_{j=1}^n |s_j-r_j|^{2H_0-2}d{\bf s}d{\bf r}  \boldsymbol {\mu} (d\boldsymbol{\xi}),\label{chaos}\\
  \le & n! \bigg(\int_{[0,t]^{n}}\Big(\int_{\R^{n}} |\F g^h_n({\bf s},\cdot, t,x)(\boldsymbol\xi)|^2 \boldsymbol {\mu}(d\boldsymbol{\xi})\Big)^{\frac1{2H_0}}d{\bf s}\bigg)^{2H_0}\notag\notag\\
 =& \frac1{n!} \bigg(\int_{[0,t]^{n}}\Big(\int_{\R^{n}} \prod_{j=1}^n \exp\Big[-[s_{\rho(j+1)}-s_{\rho(j)}]|\xi_{\rho(1)}+\cdots+ \xi_{\rho(j)}|^2\Big] \boldsymbol {\mu}(d\boldsymbol{\xi})\Big)^{\frac1{2H_0}}d{\bf s}\bigg)^{2H_0}\notag\\
 =& (n!)^{2H_0-1} \bigg(\int_{[0,t]_<^{n}}\Big(\int_{\R^{n}} \prod_{j=1}^n \exp\big[-[s_{j+1}-s_{j}]|\xi_{1}+\cdots+ \xi_{j}|^2\big] \boldsymbol {\mu}(d\boldsymbol{\xi})\Big)^{\frac1{2H_0}}d{\bf s}\bigg)^{2H_0}\notag\\
 =& (n!)^{2H_0-1} \bigg(\int_{[0,t]_<^{n}}\Big(\int_{\R^{n}} \prod_{j=1}^n \exp\Big[-[s_{j+1}-s_{j}]|\eta_{j}|^2\Big] |\eta_j-\eta_{j-1}|^{1-2H}d\boldsymbol{\eta}\Big)^{\frac1{2H_0}}d{\bf s}\bigg)^{2H_0} .
    \end{align*}
  Then, using similar argument as in the proof of Theorem \ref{thm1} and Lemma \ref{lem1}, we have
\begin{align*}
 & n! \|g^h_n(\cdot, t,x)\|^2_{\H^{\otimes n}}\\
 \le & (n!)^{2H_0-1} \bigg(\int_{[0,t]_<^{n}}\Big( \sum_{\alpha\in \mathcal A_n} \int_{\R^{n}} \prod_{j=1}^n \exp\big[-(s_{j+1}-s_{j})|\eta_{j}|^2\big] |\eta_j|^{(1-2H)\alpha_j}d\boldsymbol{\eta}\Big)^{\frac1{2H_0}}d{\bf s}\bigg)^{2H_0}\\
 \le & C^n (n!)^{2H_0-1} \left(\int_{[0,t]_<^{n}}\sum_{\alpha\in \mathcal A_n} \prod_{j=1}^n  (s_{j+1}-s_j)^{-\frac1{4H_0}[1+(1-2H)\alpha_j]}  d{\bf s}\right)^{2H_0}.
    \end{align*}
For each fixed $\alpha\in \mathcal A_n$, denote $\beta_j= \frac1{4H_0}[1+(1-2H)\alpha_j]\in (0, 1)$ noting that  $H_0\in(\frac12,1)$, $H\in(0,\frac12)$ and $H_0+H>3/4$, and then $\beta=\sum_{j=1}^n \beta_j=\frac{n(1-H)}{2H_0}$.  By Lemma \ref{lem2} we have 
\[\int_{[0,t]_<^{n}}  \prod_{j=1}^n  (s_{j+1}-s_j)^{-\beta_j}   d{\bf s}=\frac{\prod_{i=1}^n\Gamma(1-\beta_i)t^{n-\beta}}{\Gamma(n+1-\beta)}. \]
Therefore, since $\#(\mathcal A_n)=2^{n-1}$, there exists some positive constant $C$ such that,
 \begin{align*}
& n! \|g^h_n(\cdot, t,x)\|^2_{\H^{\otimes n}}\\
\le & C^n (n!)^{2H_0-1} \left(\frac{t^{n-\beta}}{\Gamma(n+1-\beta)}\right)^{2H_0}= C^n (n!)^{2H_0-1}\frac{1}{\Gamma(\frac{(2H_0+H-1)n}{2H_0}+1)^{2H_0}} t^{n(2H_0+H-1)}.
\end{align*}
It follows from Lemma \ref{lem22} and a similar argument in dealing with \eqref{eq-2.7} that there exists a positive constant $C$ such that 
\[ \E[|u^h(t,x)|^2]=\sum_{n=0}^\infty n! \|g^h_n(\cdot, t,x)\|^2_{\H^{\otimes n}}\le C\exp\left(Ct^{\frac{2H_0+H-1}{H}}\right).\]

\end{remark}

\begin{proposition}\label{prop3.4}
When  $H_0=\frac12$ and $d=1$, the condition $\kappa>3-4H$ is also a necessary condition for the existence of the square integrable solutions to \eqref{swe}. When  $H_0=\frac12$ and $H_j<\frac12, j=1, \cdots, d$, the equation \eqref{swe} has a solution only if $d=1$. 
\end{proposition}
\begin{proof}
When $d=1$, the $L^2$-norm of the second chaos of the solution is 
\begin{align*}
&\|g_2(\cdot, t,x)\|^2_{\mathcal H^{\otimes 2}}\\
  =&2 \int_{[0,t]_<^2}\int_{\R^{2}}   \frac{\sin^2((s_{2}-s_1)|\eta_{1}|^{\kappa/2})}{|\eta_1|^\kappa}\frac{\sin^2((t-s_2)|\eta_{2}|^{\kappa/2})}{|\eta_2|^\kappa}  |\eta_1|^{1-2H}|\eta_2-\eta_1|^{1-2H} d \boldsymbol{\eta}d{\bf s}\\
  \ge & \int_{[0,t]_<^2}\int_{\R_+^{2}}   \frac{\sin^2((s_{2}-s_1)|\eta_{1}|^{\kappa/2})}{|\eta_1|^\kappa}\frac{\sin^2((t-s_2)|\eta_{2}|^{\kappa/2})}{|\eta_2|^\kappa}|\eta_1|^{1-2H}|\eta_2+\eta_1|^{1-2H}d \boldsymbol{\eta}d{\bf s}\\
  \ge & \int_{[0,t]_<^2}\int_{\R_+^{2}}   \frac{\sin^2((s_{2}-s_1)|\eta_{1}|^{\kappa/2})}{|\eta_1|^\kappa}\frac{\sin^2((t-s_2)|\eta_{2}|^{\kappa/2})}{|\eta_2|^\kappa} |\eta_1|^{2(1-2H)}d \boldsymbol{\eta}d{\bf s},
\end{align*}
where the last integral is infinity if $\kappa\leq 3-4H$ due to Lemma \ref{lem5}.

For general dimension $d$, the above estimation becomes
\begin{align*}
&\|g_2(\cdot, t,x)\|^2_{\mathcal H^{\otimes 2}}\\
  =&2 \int_{[0,t]_<^2}\int_{\R^{2d}}   \frac{\sin^2((s_{2}-s_1)|\eta_{1}|^{\kappa/2})}{|\eta_1|^\kappa}\frac{\sin^2((t-s_2)|\eta_{2}|^{\kappa/2})}{|\eta_2|^\kappa}  \prod_{j=1}^d|\eta_1^j|^{1-2H_j}|\eta_2^j-\eta_1^j|^{1-2H_j}d \boldsymbol{\eta}d{\bf s}\\
  \ge & \int_{[0,t]_<^2}\int_{\R_+^{2d}}   \frac{\sin^2((s_{2}-s_1)|\eta_{1}|^{\kappa/2})}{|\eta_1|^\kappa}\frac{\sin^2((t-s_2)|\eta_{2}|^{\kappa/2})}{|\eta_2|^\kappa} \prod_{j=1}^d|\eta_1^j|^{1-2H_j}|\eta_2^j+\eta_1^j|^{1-2H_j}d \boldsymbol{\eta}d{\bf s}\\
  \ge & \int_{[0,t]_<^2}\int_{\R_+^{2d}}   \frac{\sin^2((s_{2}-s_1)|\eta_{1}|^{\kappa/2})}{|\eta_1|^\kappa}\frac{\sin^2((t-s_2)|\eta_{2}|^{\kappa/2})}{|\eta_2|^\kappa} \prod_{j=1}^d |\eta_1^j|^{2(1-2H_j)}d \boldsymbol{\eta}d{\bf s}.
\end{align*}

Now,  by the change of variables \[ 
  \left\{ 
    \begin{array}{lc}
     \eta^1_1=r\cos(\theta_1) & \\
   \eta^2_1=r\sin(\theta_1)\cos(\theta_2)  & \\
  \eta^3_1=r\sin(\theta_1)\sin(\theta_2)\cos(\theta_3)  & \\
  \vdots  &\\
\eta^{d-1}_1=r\sin(\theta_1)\cdots\sin(\theta_{d-2}) \cos(\theta_{d-1})  &  \\
\eta^d_1=r\sin(\theta_1)\cdots\sin(\theta_{d-2}) \sin(\theta_{d-1}),
    \end{array}\right.
\] we have
\begin{align*}
\int_{\R^d_+} \frac{\sin^2(|\eta_1|^{\kappa/2})}{|\eta_1|^\kappa}\prod_{j=1}^d |\eta^j_1|^{2(1-2H_j)} d\eta_1
=&C_d\int_0^\infty\sin^2(r^{\kappa/2}) r^{\sum\limits_{j=1}^d 2(1-2H_j)-\kappa+d-1}dr,
\end{align*}
which by Lemma \ref{lem5} is infinite  when $d>1$ since  $\sum\limits_{j=1}^d2(1-2H_j)-\kappa+d-1\ge-1$ for $H_j\in(0,\frac12)$. \hfill
\end{proof}

\section{Moments of the solution and weak intermittency}\label{sec-moment}
In this section, we first obtain the lower bound and upper bound for the $p$-th moment of the solution to \eqref{swe} for $p\ge2$, and then deduce the weak intermittency. 

\begin{proposition}\label{prop1}
Under the conditions in Theorem \ref{thm1}, there exist  $0<C_1, C_2<\infty$ such that for all $p\ge2$
\begin{equation}\label{eq3.1}
C_1\exp\left(C_1t^{\frac{2\kappa H_0+2(\kappa-2)+4H}{3\kappa-4+4H}} \right) \le\|u(t,x)\|_p \le C_2\exp\left(C_2p^{\frac{\kappa}{3\kappa-4+4H}}t^{\frac{2\kappa H_0+2(\kappa-2)+4H}{3\kappa-4+4H}} \right)
\end{equation}
and
\begin{align}\label{eq3.1'}
{C_1}&\le\liminf_{t\to\infty} t^{-\frac{2\kappa H_0+2(\kappa-2)+4H}{3\kappa-4+4H}} \log \|u(t,x)\|_p \notag\\
&\le \limsup_{t\to\infty} t^{-\frac{2\kappa H_0+2(\kappa-2)+4H}{3\kappa-4+4H}} \log \|u(t,x)\|_p\le C_2\,p^{\frac{\kappa}{3\kappa-4+4H}}.
\end{align}
In particular, $\frac{2\kappa H_0+2(\kappa-2)+4H}{3\kappa-4+4H}=1$ if $H_0=\frac12.$
\end{proposition}
\begin{proof}
We shall prove \eqref{eq3.1}  for $H_0\in(\frac12, 1)$. The proof for $H_0=\frac12$ is similar and thus omitted.

By \eqref{chaos}, we have
\begin{align*}
  &n! \|g_n(\cdot, t,x)\|^2_{\H^{\otimes n}}\notag\\
= &  n! \int_{\R^{n}} \int_{[0,t]^{2n}} \F g_n({\bf s},
  \cdot,t,x)(\boldsymbol{\xi})\overline{\F g_n({\bf r},\cdot, t,x)(\boldsymbol{\xi})} \prod_{j=1}^n |s_j-r_j|^{2H_0-2}d{\bf s}d{\bf r}  \boldsymbol {\mu} (d\boldsymbol{\xi})\\
  =&n!\int_{\R^{n}} \int_{([0,t]_<^{n})^2} \prod_{j=1}^n \frac{\sin((s_{j+1}-s_{j})|\xi_{1}+\dots+\xi_{j}|^{\kappa/2})}{|\xi_{1}+\dots+\xi_{j}|^{\kappa/2}} \\
  &\qquad \qquad \prod_{j=1}^n \frac{\sin((r_{j+1}-r_{j})|\xi_{1}+\dots+\xi_{j}|^{\kappa/2})}{|\xi_{1}+\dots+\xi_{j}|^{\kappa/2}} \prod_{j=1}^n |s_j-r_j|^{2H_0-2}d{\bf s}d{\bf r}  \boldsymbol {\mu} (d\boldsymbol{\xi})\\
  =&n!\int_{\R^{n}}\Bigg( \int_{([0,t]_<^{n})^2} \prod_{j=1}^n \frac{\sin((s_{j+1}-s_{j})|\eta_{j}|^{\kappa/2})}{|\eta_{j}|^{\kappa/2}}  \frac{\sin((r_{j+1}-r_{j})|\eta_{j}|^{\kappa/2})}{|\eta_{j}|^{\kappa/2}} \\
  &\qquad \qquad \qquad  \prod_{j=1}^n |s_j-r_j|^{2H_0-2} ~ d{\bf s}d{\bf r}\Bigg) ~ \prod_{j=1}^n|\eta_j-\eta_{j-1}|^{1-2H}d\boldsymbol{\eta}\\
  \ge &n!\int_{\mathbb D_n}\Bigg( \int_{([0,t]_<^{n})^2} \prod_{j=1}^n \frac{\sin((s_{j+1}-s_{j})|\eta_{j}|^{\kappa/2})}{|\eta_{j}|^{\kappa/2}}  \frac{\sin((r_{j+1}-r_{j})|\eta_{j}|^{\kappa/2})}{|\eta_{j}|^{\kappa/2}} \\
  &\qquad \qquad \qquad  \prod_{j=1}^n |s_j-r_j|^{2H_0-2} ~ d{\bf s}d{\bf r}\Bigg) ~ \prod_{j=1}^n|\eta_j|^{1-2H}d\boldsymbol{\eta},
    \end{align*}
    noting that in the last step we used the facts  that the inner integral with respect to $d{\bf s}d{\bf r}$ is nonnegative and that $|\eta_j-\eta_{j-1}|^{1-2H}\ge |\eta_j|^{1-2H}$ on $\mathbb D_n$ with $\mathbb D_n=\{(\eta_1, \dots, \eta_n)\in \R^n: \eta_1\ge0, \eta_2\le0, \eta_3\ge 0,\eta_4\le 0, \dots\}$.
  
  Now, we have 
  \begin{align}\label{g-A}
  &n! \|g_n(\cdot, t,x)\|^2_{\H^{\otimes n}}\notag\\
   \ge &\ n! \int_{([0,t]_<^{n})^2} \Bigg( \int_{\mathbb D_n} \prod_{j=1}^n \sin((s_{j+1}-s_{j})|\eta_{j}|^{\kappa/2}) \sin((r_{j+1}-r_{j})|\eta_{j}|^{\kappa/2}) \prod_{j=1}^n|\eta_j|^{1-2H-\kappa}d\boldsymbol{\eta} \Bigg)\notag\\
   &\qquad \qquad \qquad  \prod_{j=1}^n |s_j-r_j|^{2H_0-2} ~ d{\bf s}d{\bf r}\notag\\
  = &\ n! \int_{([0,t]_<^{n})^2} \Bigg(\prod_{j=1}^n \int_{\R_+}  \sin((s_{j+1}-s_{j})|\eta|^{\kappa/2}) \sin((r_{j+1}-r_{j})|\eta|^{\kappa/2}) |\eta|^{1-2H-\kappa}d\eta \Bigg)\notag\\
   &\qquad \qquad \qquad  \prod_{j=1}^n |s_j-r_j|^{2H_0-2} ~ d{\bf s}d{\bf r}\notag\\
     \ge &\ n! t^{n(2H_0-2)}  \int_{\R_+^n} \left|\prod_{j=1}^n\int_{[0,t]_<^{n}} \prod_{j=1}^n\sin((s_{j+1}-s_j)|\eta_j|^{\kappa/2}) d{\bf s}\right|^2\boldsymbol \prod_{j=1}^n|\eta_j|^{1-2H-\kappa} d\boldsymbol{\eta},
  \end{align}
  where the last step holds due to $|s_j-r_j|\le t$ and the fact that the integral with respect to $\eta$ is nonnegative by Lemma \ref{lem6}.

%
 { Let 
 \begin{equation}\label{A}
 A_n(t)=\int_{\R_+^n} \left|\int_{[0,t]_<^{n}} \prod_{j=1}^n\sin((s_{j+1}-s_j)|\eta_j|^{\kappa/2}) d{\bf s}\right|^2\boldsymbol \prod_{j=1}^n|\eta_j|^{1-2H-\kappa} d\boldsymbol{\eta}.
 \end{equation}
Make the change of variables $s_j'=s_j/t$ and $\eta_j'=\eta_jt^{2/\kappa}$, and we have the scaling 
\begin{equation}\label{eq-scaling}
A_n(t)=t^{4n(1-\frac{1-H}\kappa)} A_n(1).
\end{equation}
Now we estimate $\E[A_n(\tau)]$ where $\tau$ is an exponential random time with parameter $1$. By Fubini's Theorem and  Jensen's inequality, we obtain
\begin{align}\label{a-tau-1}
&\E[A_n(\tau)]=\int_0^\infty e^{-t} A_n(t) dt\notag\\
=&\int_{\R_+^n}\int_0^\infty e^{-t} \left|\int_{[0,t]_<^{n}} \prod_{j=1}^n\sin((s_{j+1}-s_j)|\eta_j|^{\kappa/2}) d{\bf s}\right|^2dt~~\boldsymbol \prod_{j=1}^n|\eta_j|^{1-2H-\kappa} d\boldsymbol{\eta}\notag\\
\ge & \int_{\R_+^n}\left|\int_0^\infty e^{-t} \int_{[0,t]_<^{n}} \prod_{j=1}^n\sin((s_{j+1}-s_j)|\eta_j|^{\kappa/2}) d{\bf s}dt\right|^2~~\boldsymbol \prod_{j=1}^n|\eta_j|^{1-2H-\kappa} d\boldsymbol{\eta}. 
\end{align}
 Applying the change of variables $r_j=s_{j+1}-s_{j}$ for $j=0, 1,\dots, n$ with the convention $s_0=0$ and $s_{n+1}=t$ and using Lemma \ref{lem8},  we have
\begin{align}\label{a-tau-2}
&\int_0^\infty e^{-t} \int_{[0,t]_<^{n}} \prod_{j=1}^n\sin((s_{j+1}-s_j)|\eta_j|^{\kappa/2}) d{\bf s}dt\notag\\
=&\int_{\R_+^{n+1}} e^{-(r_0+r_1+\dots+r_n)}  \prod_{j=1}^n\sin(r_j |\eta_j|^{\kappa/2}) dr_0dr_1\dots dr_n\notag\\
=&\prod_{j=1}^n \frac{|\eta_j|^{\kappa/2}}{1+|\eta_j|^\kappa}.
\end{align}
Now, combining \eqref{a-tau-1} and \eqref{a-tau-2}, we get
\begin{align*}
\E[A_n(\tau)]\ge \int_{\R_+^n } \prod_{j=1}^n \frac{|\eta_j|^{1-2H}}{(1+|\eta_j|^\kappa)^2}d\boldsymbol{\eta}=\left(\int_{\R_+} \frac{|\eta|^{1-2H}}{(1+|\eta|^\kappa)^2}d\eta\right)^n=c^n,
\end{align*}
where $c=\int_{\R_+} \frac{|\eta|^{1-2H}}{(1+|\eta|^\kappa)^2}d\eta\in (0,\infty)$.
Together with the scaling  property \eqref{eq-scaling}, we have 
\begin{align}\label{eq-A-1}
c^n\le \E[A_n(\tau)]= \E[\tau^{4n(1-\frac{1-H}{\kappa})}]A_n(1).
\end{align}
Therefore, it implies from \eqref{g-A}, \eqref{A}, \eqref{eq-scaling}, \eqref{eq-A-1} and the fact $\E[\tau^x]=\Gamma(x+1)$ that
\begin{align*}
n! \|g_n(\cdot, t,x)\|^2_{\H^{\otimes n}}\ge&\ n!t^{n(2H_0-2)} A_n(t)\notag\\
=&\ n!t^{n(2H_0-2)} t^{4n(1-\frac{1-H}{\kappa})}A_n(1)\\
\ge&\ C^n n!t^{n(2H_0-2)} t^{4n(1-\frac{1-H}{\kappa})}\frac1{\E[\tau^{4n(1-\frac{1-H}{\kappa})}]}\\
=&\ C^n t^{n(2H_0+2-\frac{4-4H}{\kappa})}\frac{n!}{\Gamma(4n(1-\frac{1-H}{\kappa})+1)}\\
\sim &\ \frac{C^n t^{n(2H_0+2-\frac{4-4H}{\kappa})}}{(n!)^{3-\frac{4(1-H)}{\kappa}}a^{an+\frac12} n^{\frac{1-a}{2}}}\ ,
\end{align*}
where the last step follows from \eqref{gamma-appr} in Lemma \ref{lem22} with $a=4\left(1-\frac{1-H}{\kappa}\right)$ and $b=1$.  Noting that there exists $\lambda>1$ such that $\lambda^{-n}\leq a^{an+\frac12} n^{\frac{1-a}{2}}\leq \lambda^n$, we obtain
\[
n! \|g_n(\cdot, t,x)\|^2_{\H^{\otimes n}}\ge\ \frac{C^n t^{n(2H_0+2-\frac{4-4H}{\kappa})}}{(n!)^{3-\frac{4(1-H)}{\kappa}}}.
\]
Therefore, applying \eqref{power-est} in Lemma \ref{lem22}, we have
\[
\Vert u(t,x)\Vert_p\geq \Vert u(t,x)\Vert_2\geq \left(\sum_{n=0}^\infty \frac{C^n t^{n(2H_0+2-\frac{4-4H}{\kappa})}}{(n!)^{3-\frac{4(1-H)}{\kappa}}}\right)^{\frac12}\geq C_1\exp{\left(C_1t^{\frac{2\kappa H_0+2(\kappa-2)+4H}{3\kappa-4+4H}}\right)},
\]
for some $C_1>0$.}

For the upper bound, noting that $\|I_n(g_n)\|_p\le (p-1)^{\frac n2}\|I_n(g_n)\|_2$  (see the last line on Page 62 in \cite{nualart}), by Minkowski's inequality and similar arguments in \eqref{eq-2.7}-\eqref{eq2.7} we have
\begin{align*}
\|u\|_p\le \sum_{n=0}^\infty \|I_n(g_n)\|_p&\le \sum_{n=0}^\infty (p-1)^{\frac n2} \|I_n(g_n)\|_2\\
&\le \sum_{n=0}^\infty  \frac{(p-1)^{\frac n2}C^{\frac n2}t^{n\left(H_0+\left(1-\frac2\kappa\right)+\frac{2H}\kappa\right)}}{(n!)^{\left[\left(1-\frac{2}{\kappa}\right)+\frac{2H}{\kappa}\right]+\frac12}}\\
& \le C\exp\left(Cp^{\frac{\kappa}{3\kappa-4+4H}} t^{\frac{2\kappa H_0+2(\kappa-2)+4H}{3\kappa-4+4H}}\right),
\end{align*}
and hence, we get
\[
\|u\|_p\le C_2\exp\left(C_2p^{\frac{\kappa}{3\kappa-4+4H}} t^{\frac{2\kappa H_0+2(\kappa-2)+4H}{3\kappa-4+4H}}\right),
\]
for some $C_2>0$. 
\hfill
\end{proof}

\begin{remark}\label{remark2}
  
 Note that the lower and upper bounds (lower bounds, resp.) for the $p$-th moment  were studied  in \cite{dm09} (in \cite{bc16}, resp.) based on the probabilistic representation for the second moment of the solution to wave equations obtained in \cite{dmt}. 
\end{remark}

Recall that the lower  Lyapunov exponent $L_l(p)$ and the upper Lyapunov exponent $L_u(p)$ of order $p\ge 2$ of the solution $u(t,x)$  are defined, respectively by 
\[L_l(p)=\liminf_{t\to\infty} \frac1{R(t)} \inf_{x\in\R} \log \E[|u(t,x)|^p]\]
and 
\[L_u(p)=\limsup_{t\to\infty} \frac1{R(t)} \sup_{x\in\R} \log \E[|u(t,x)|^p]\]
for some positive function $R(t).$ If $L_l(2)>0$ and $L_u(p)<\infty$ for all $p\ge 0$, we say that $u(t,x)$ possesses the weak intermittency. Heuristically speaking, if a process $u(t,x)$ is weakly intermittent, it concentrates  on a few of very high peaks (see, e.g., \cite{davar} and the references therein).  Taking $R(t)=t^{\frac{2\kappa H_0+2(\kappa-2)+4H}{3\kappa-4+4H}}$, the Proposition below follows directly from Proposition \ref{prop1}.
\begin{proposition}
Under the conditions in Theorem \ref{thm1}, the solution $u(t,x)$ to \eqref{swe} is weakly intermittent. 
\end{proposition}

\section{H\"older continuity}\label{sec-continuity}
In this section, the H\"older continuity in time and space for the solution $u(t,x)$ to the SWE \eqref{swe} is obtained in Proposition \ref{prop2}.  The result is consistent with the H\"older continuity for SWEs with noise that is not rough in space obtained in  \cite{balan12}, \cite{bs17}, and \cite{cd08}. Note that exponents of H\"older continuity in both time and space are independent of the temporal covariance function $f_0(t)$. Similar phenomenon occurs also for SHEs (see \cite[Theorem 3.2]{bqs}).

\begin{proposition} \label{prop2}
Assume the same conditions as in  Theorem \ref{thm1}. 
Then on any  set $[0,T]\times \mathbb M$ where $\mathbb M\subset \R$ is a compact set,  $u(t,x)$ has a modification which is $\theta_1$-H\"older continuous  in time for all $\theta_1\in(0, 1-\frac2\kappa+\frac{2H}{\kappa})$ and $\theta_2$-H\"older continuous in space for all $\theta_2\in(0, H+\frac\kappa2-1)$. 

In particular, when $\kappa=2,$ the solution has a version that is $\theta$-H\"oder continuous both in time and in space for all $\theta\in(0, H).$
\end{proposition}
\begin{proof}
First we show the H\"older continuity in space.
Noting that $\|I_n(g_n)\|_p\le (p-1)^{\frac n2}\|I_n(g_n)\|_2$, by Minkowski's inequality, we have 
\begin{align*}
&\|u(t,x+z)-u(t,x)\|_p\\
 \le&{ \sum_{n=1}^\infty} (p-1)^{\frac n2} (n!)^{\frac12}\| g_n(\cdot, t, x+z)-g_n(\cdot, t, x)\|_{\H^{\otimes n}}\\
 =& { \sum_{n=1}^\infty} (p-1)^{\frac n2} \bigg(n! \int_{\R^{n}} \int_{[0,t]^{2n}} \F \Big[g_n({\bf s},
  \cdot,t,x+z)(\boldsymbol{\xi})-g_n({\bf s},
  \cdot,t,x)(\boldsymbol{\xi})\Big]\\
  &\qquad\qquad  \overline{\F\Big[ g_n({\bf r},\cdot, t,x+z)(\boldsymbol{\xi}) -g_n({\bf r},\cdot, t,x)(\boldsymbol{\xi})\Big]} \prod_{j=1}^n |s_j-r_j|^{2H_0-2}d{\bf s}d{\bf r}  \boldsymbol {\mu}(d\boldsymbol{\xi}) \bigg)^\frac12\\
 \le & { \sum_{n=1}^\infty} (p-1)^{\frac n2}(n!)^{H_0-\frac12} \bigg( \int_{[0,t]_<^{n}}  \Big[ \int_{\R^n} \prod_{j=1}^n\frac{\sin^2((s_{j+1}-s_j)|\xi_1+\dots+\xi_j|^{\kappa/2})}{|\xi_1+\dots+\xi_j|^\kappa}|\xi_j|^{1-2H} \\
 &\qquad \qquad\qquad \qquad  \qquad \qquad \left|1-e^{-iz(\xi_1+\dots+\xi_n)}\right|^2 d {\boldsymbol \xi}\Big]^{\frac1{2H_0}}d{\bf s}  \bigg)^{H_0}\\
=& { \sum_{n=1}^\infty}  (p-1)^{\frac n2}(n!)^{H_0-\frac12} \bigg( \int_{[0,t]_<^{n}}  \Big[ \int_{\R^n} \prod_{j=1}^n\frac{\sin^2((s_{j+1}-s_j)|\eta_j|^{\kappa/2})}{|\eta_j|^\kappa}|\eta_j-\eta_{j-1}|^{1-2H} \\
 &\qquad \qquad\qquad \qquad  \qquad \qquad \left|1-e^{-iz\eta_n}\right| d {\boldsymbol \eta}\Big]^{\frac1{2H_0}}d{\bf s}  \bigg)^{H_0}.
\end{align*}
Now, by changing of variables, we have
 \begin{align}\label{holder-x-1}
& \int_{[0,t]_<^{n}}  \bigg[ \int_{\R^n} \prod_{j=1}^n\frac{\sin^2((s_{j+1}-s_j)|\eta_j|^{\kappa/2})}{|\eta_j|^\kappa}|\eta_j-\eta_{j-1}|^{1-2H} \left|1-e^{-iz\eta_n}\right| d {\boldsymbol \eta}\bigg]^{\frac1{2H_0}}d{\bf s} d{\bf t} \notag\\
 \le & \int_{[0,t]_<^{n}}  \bigg[ \sum_{\alpha\in \mathcal A_n}\int_{\R^n} \prod_{j=1}^n\frac{\sin^2((s_{j+1}-s_j)|\eta_j|^{\kappa/2})}{|\eta_j|^\kappa}|\eta_j|^{\alpha_j(1-2H)} (|z\eta_n|\wedge 2) d {\boldsymbol \eta}\bigg]^{\frac1{2H_0}}d{\bf s} d{\bf t} \notag\\
\le & \int_{[0,t]_<^{n}}\sum_{\alpha\in\mathcal A_n} \left(\frac{2}{\kappa}\right)^{\frac{n}{2H_0}}\Bigg[ \prod_{j=1}^{n-1}  (s_{j+1}-s_j)^{\frac1{2H_0}[2-\frac2\kappa-\frac2\kappa   \alpha_j(1-2H)]} {\left(\int_{\R}\frac{\sin^2(\eta)}{\eta^2} |\eta|^{\frac{2}{\kappa}\alpha_j(1-2H)+\frac{2}{\kappa}-1}d\eta\right)^{\frac1{2H_0}}}\notag\\
    &\qquad \qquad \qquad {\left(\int_{\R}\frac{\sin^2((t-s_n)y)}{y^2} |y|^{\frac{2}{\kappa}\alpha_n(1-2H)+\frac{2}{\kappa}-1}((|z||y|^{\frac{2}{\kappa}})\wedge 2)dy\right)^{\frac1{2H_0}}\Bigg]d{\bf s}}.
 \end{align}
 Recall that $\alpha_n\in\{0, 1\}$, and by Lemma \ref{lem3} with $\lambda=\frac2\kappa$, $\beta=\frac2\kappa(\alpha_n(1-2H)+1)-1,\gamma=0$ for $\alpha_n=1$ and $\gamma=\frac{1-2H}{\kappa}$ for $\alpha_n=0$, we have 
 \begin{align*}\label{holder-x-2}
\int_{\R}\frac{\sin^2((t-s_n)y)}{y^2} |y|^{\frac{2}{\kappa}\alpha_n(1-2H)+\frac{2}{\kappa}-1}((|z||y|^{\frac2\kappa})\wedge 2)dy&\le C(1\vee (t-s_n)^{\frac{2(1-2H)}{\kappa}})|z|^{2H+\kappa-2}\notag \\
&{\le C|z|^{2H+\kappa-2}.}
 \end{align*} 
 {Thus, there exists a positive constant $C$ depending only on $(p,H_0, H, T)$ such that 
 \[
 \|u(t,x+z)-u(t,x)\|_p\leq  \sum_{n=1}^\infty C^{n} (n!)^{H_0-\frac12}  \left(\int_{[0,t]_<^{n}}\sum_{\alpha\in\mathcal A_n}\prod_{j=1}^{n-1}(s_{j+1}-s_j)^{\frac1{2H_0}[2-\frac2\kappa-\frac2\kappa   \alpha_j(1-2H)]}d{\bf s}\right)^{H_0}.
 \]
 
When $n=1$, the integral on the right-hand side of the above inequality equals $t$.

When $n\geq 2$, for each fixed $\alpha\in \mathcal A_n$, denote $
\beta_j=\frac1{2H_0}\left[2-\frac2\kappa-\frac2\kappa   \alpha_j(1-2H)\right],\  j=1, \dots, n$. Noting that $\alpha_n\in \{0,1\}$ and $\sum_{j=1}^n\alpha_j=n$, we have
\begin{equation}\label{beta-n-1}
\sum_{j=1}^{n-1}\beta_j=
\begin{cases}
\frac{(\kappa-2+2H)n}{H_0\kappa}-\frac{\kappa-2+2H}{H_0\kappa}, \mbox{ when } \alpha_n=1,\\
\frac{(\kappa-2+2H)n}{H_0\kappa}-\frac{\kappa-1}{H_0\kappa}, \mbox{ when } \alpha_n=0.
\end{cases}
\end{equation}
 Using Lemma \ref{lem2}, we obtain
\begin{align*}
\int_{[0,t]_<^{n}}\prod_{j=1}^{n-1}  (s_{j+1}-s_j)^{\beta_j}d{\bf s}=&\ \int_0^t\int_{0<s_1<\dots<s_n}\prod_{j=1}^{n-1}  (s_{j+1}-s_j)^{\beta_j}d{\bf s}\\
=&\ \frac{\prod_{j=1}^{n-1}\Gamma(1+\beta_j)}{\Gamma(\beta_1+\dots+\beta_{n-1}+n)}\int_0^ts_n^{\beta_1+\dots+\beta_{n-1}+n-1}ds_n\\
=&\ \frac{\prod_{j=1}^{n-1}\Gamma(1+\beta_j)}{(\beta_1+\dots+\beta_{n-1}+n)\Gamma(\beta_1+\dots+\beta_{n-1}+n)}t^{\beta_1+\dots+\beta_{n-1}+n}\\
=&\ \frac{\prod_{j=1}^{n-1}\Gamma(1+\beta_j)}{\Gamma(\beta_1+\dots+\beta_{n-1}+n+1)}t^{\beta_1+\dots+\beta_{n-1}+n}.
\end{align*}
Then applying \eqref{gamma-appr} in Lemma \ref{lem22} with $a=1+\frac{\kappa-2+2H}{H_0\kappa}\in(1,2)$ and either $b=1-\frac{\kappa-2+2H}{H_0\kappa}\in(0,1)$ or $b=1-\frac{\kappa-1}{H_0\kappa}\in[0,1)$, we have
\begin{align*}
\int_{[0,t]_<^{n}}\prod_{j=1}^{n-1}  (s_{j+1}-s_j)^{\beta_j}d{\bf s}\leq&\ \frac{C^{n-1}}{\Gamma(\beta_1+\dots+\beta_{n-1}+n+1)}t^{\beta_1+\dots+\beta_{n-1}+n}\\
=&\ \frac{C^{n-1}}{\Gamma(an+b)}t^{a(n-1)+\frac{\kappa-2+2H}{H_0\kappa}+b}\\
\sim&\ \frac{C^{n-1}}{(n!)^aa^{an+b-\frac12}n^{b-\frac12-\frac{a}{2}}}t^{a(n-1)+\frac{\kappa-2+2H}{H_0\kappa}+b}\\
\leq&\ \frac{C^{n-1}T^{\frac{\kappa-2+2H}{H_0\kappa}+b}t^{a(n-1)}}{(n!)^a}.
\end{align*}
 Therefore, there exists a positive constant $C$ depending only on $(p,H_0, H, T)$ such that 
\begin{align}\label{holder-x-3}
\|u(t,x+z)-u(t,x)\|_p
\le & |z|^{H+\frac{\kappa}{2}-1}\left((p-1)^{\frac12}Ct+\sum_{n=2}^\infty \frac{(p-1)^{\frac n2} C^{n-1}t^{\frac{(H_0\kappa+\kappa-2+2H)(n-1)}{\kappa}}}{(n!)^{\frac{3\kappa-4+4H}{2\kappa}}}\right)\notag\\
\leq &\ C |z|^{H+\frac{\kappa}{2}-1}\sum_{n=0}^\infty \frac{C^nt^{\frac{(H_0\kappa+\kappa-2+2H)n}{\kappa}}}{(n!)^{\frac{3\kappa-4+4H}{2\kappa}}}\leq C |z|^{H+\frac{\kappa}{2}-1}.
\end{align}}
Then the $\theta_2$-H\"older continuity for $\theta_2\in(0, H+\frac\kappa2-1)$ follows from the Kolmogorov's continuity criterion. 

Now we consider the H\"older continuity in time. 
\begin{align}\label{eq4.1}
\|u(t+h, x)-u(t, x)\|_p
\le& { \sum_{n=1}^\infty}(p-1)^{\frac n2}\|g_n(\cdot, t+h, x)-g_n(\cdot, t,x)\|_{\H^{\otimes n}}\notag\\
{\le}&  \sum_{n=1}^\infty(p-1)^{\frac n2}(n!)^{\frac12} \left[ \sqrt{A_n(t,h)}+\sqrt{B_n(t,h)}\right],
\end{align}
where
\[A_n(t,h)=\|g_n(\cdot, t+h, x)I_{[0,t]^n}-g_n(\cdot, t, x)\|^2_{\H^{\otimes n}}\]
and
\[B_n(t,h)=\|g_n(\cdot, t+h, x)I_{[0,t+h]^n\backslash [0,t]^n}\|^2_{\H^{\otimes n}}.\]
For $A_n(t,h)$, we have
\begin{align*}
&A_n(t,h)=\|g_n(\cdot, t+h, x)I_{[0,t]^n}-g_n(\cdot, t, x)\|^2_{\H^{\otimes n}}\\
=&\int_{\R^{n}} \int_{[0,t]^{2n}} \F \Big[g_n({\bf s},
  \cdot,t+h,x)(\boldsymbol{\xi})-g_n({\bf s},
  \cdot,t,x)(\boldsymbol{\xi})\Big]\\
  &\qquad\qquad  \overline{\F\Big[ g_n({\bf r},\cdot, t+h,x)(\boldsymbol{\xi}) -g_n({\bf r},\cdot, t,x)(\boldsymbol{\xi})\Big]} \prod_{j=1}^n |s_j-r_j|^{2H_0-2}d{\bf s}d{\bf r}  \boldsymbol {\mu}(d\boldsymbol{\xi})\\
  \le&(n!)^{2H_0-2} \Bigg(\int_{[0,t]_<^{n}}  \Bigg[ \int_{\R^n} \prod_{j=1}^{n-1}\frac{\sin^2((s_{j+1}-s_j)|\xi_1+\dots+\xi_j|^{\kappa/2})}{|\xi_1+\dots+\xi_j|^\kappa}\prod_{j=1}^n|\xi_j|^{1-2H} \\
 & \qquad \qquad  \Bigg|\frac{\sin((t+h-s_n)|\xi_1+\dots+\xi_n|^{\kappa/2})}{|\xi_1+\dots+\xi_n|^{\kappa/2}} - \frac{\sin((t-s_n)|\xi_1+\dots+\xi_n|^{\kappa/2})}{|\xi_1+\dots+\xi_n|^{\kappa/2}}  \Bigg|^2d {\boldsymbol \xi}\Bigg]^{\frac1{2H_0}}d{\bf s} \Bigg)^{2H_0}\\
 =&(n!)^{2H_0-2} \Bigg(\int_{[0,t]_<^{n}}  \Bigg[ \int_{\R^n} \left(\prod_{j=1}^{n-1}\frac{\sin^2((s_{j+1}-s_j)|\eta_j|^{\kappa/2})}{|\eta_j|^\kappa}\right)\prod_{j=1}^n|\eta_j-\eta_{j-1}|^{1-2H} \\
 & \qquad\qquad  \qquad \frac{|\sin((t+h-s_n)|\eta_n|^{\kappa/2})-\sin((t-s_n)|\eta_n|^{\kappa/2})|^2}{|\eta_n|^\kappa} d {\boldsymbol \eta}\Bigg]^{\frac1{2H_0}}d{\bf s} \Bigg)^{2H_0}\\
 \le &(n!)^{2H_0-2} \Bigg(\int_{[0,t]_<^{n}}  \sum_{\alpha\in\mathcal A_n}\Bigg[ \int_{\R^n} \left(\prod_{j=1}^{n-1}\frac{\sin^2((s_{j+1}-s_j)|\eta_j|^{\kappa/2})}{|\eta_j|^\kappa}\right)\prod_{j=1}^n|\eta_j|^{\alpha_j(1-2H)}\\
 & \qquad \qquad \qquad \frac{C^2_\gamma((|h|^{2\gamma}|\eta_n|^{\gamma\kappa})\wedge (|h|^2|\eta_n|^\kappa))}{|\eta_n|^\kappa} d {\boldsymbol \eta}\Bigg]^{\frac1{2H_0}}d{\bf s}  \Bigg)^{2H_0},
\end{align*}
where the last step follows from Lemma \ref{lem66} with $\gamma\in(0, 1-\frac2\kappa+\frac{2H}{\kappa})$. Note that when $\gamma\in(0, 1-\frac2\kappa+\frac{2H}{\kappa})$, 
\[\int_{\R} \frac{((|h|^{2\gamma}|\eta_n|^{\gamma\kappa})\wedge (|h|^2|\eta_n|^\kappa))}{|\eta_n|^\kappa} |\eta_n|^{\alpha_n(1-2H)} d\eta_n \leq C(|h|^2+|h|^{2\gamma}) \text{ for } \alpha_n\in\{0,1\} \]  and following the approach { in the analysis of \eqref{holder-x-1}-\eqref{holder-x-3}}, we can show that  
\begin{equation}\label{eq4.2}
\sum_{n=1}^\infty(p-1)^{\frac n2}(n!)^{\frac12}  \sqrt{A_n(t,h)}\le  C (|h|+|h|^{\gamma}){ \leq C|h|^{\gamma}}
\end{equation}
for $\gamma\in(0, 1-\frac2\kappa+\frac{2H}{\kappa})$ with $C$ depending on $ (p,\kappa, H_0, H, T, \mathbb M, \gamma)$. 

Now we consider the term $B_n(t,h)$.  Denote $E_{t,h}=[0,t+h]^n\backslash [0,t]^n$, and then
{\[
E_{t,h}=\bigcup_{\rho\in\mathcal P_n}\Big\{(s_1, \dots, s_n): s_{\rho(1)}\le s_{\rho(2)}\le \cdots\le s_{\rho(n)}, t<s_{\rho(n)}\le t+h \Big\}.
\]}
Therefore, we have
\begin{align*}
&B_n(t,h)=\|g_n(\cdot, t+h, x)I_{[0,t+h]^n\backslash [0,t]^n}\|^2_{\H^{\otimes n}}\\
=&\,\int_{\R^{n}} \int_{[0,t+h]^{2n}} \F \Big[g_n({\bf s},
  \cdot,t+h,x)(\boldsymbol{\xi})\Big] \overline{\F\Big[ g_n({\bf r},\cdot, t+h,x)(\boldsymbol{\xi})\Big]} \\
  &\qquad \qquad I_{E_{t, h}}({\bf s})I_{E_{t,h}}({\bf r})\prod_{j=1}^n |s_j-r_j|^{2H_0-2}d{\bf s}d{\bf r}  \boldsymbol {\mu}(d\boldsymbol{\xi})\\
 \le & (n!)^{-2}
  \Bigg( \sum_{\rho\in\mathcal P_n}\int_t^{t+h} \int_{0<s_{\rho(1)}<s_{\rho(2)}<\cdots<s_{\rho(n)}} \bigg[ \int_{\R^n} \prod_{j=1}^{n}\frac{\sin^2((s_{\rho(j+1)}-s_{\rho(j)})|\xi_{\rho(1)}+\dots+\xi_{\rho(j)}|^{\kappa/2})}{|\xi_{\rho(1)}+\dots+\xi_{\rho(j)}|^\kappa}\\
  &\qquad \qquad \qquad\qquad  \prod_{j=1}^n|\xi_j|^{1-2H} d{\boldsymbol \xi}\bigg]^{\frac1{2H_0}}d{\bf s} \Bigg)^{2H_0}\\
  = &\, (n!)^{-2}
  \Bigg(n!\int_t^{t+h} \int_{0<s_1<s_2<\cdots<s_n} \bigg[ \int_{\R^n} \prod_{j=1}^{n}\frac{\sin^2((s_{j+1}-s_j)|\xi_1+\dots+\xi_j|^{\kappa/2})}{|\xi_1+\dots+\xi_j|^\kappa}\prod_{j=1}^n|\xi_j|^{1-2H} d{\boldsymbol \xi}\bigg]^{\frac1{2H_0}}d{\bf s} \Bigg)^{2H_0}\\
=&\, {(n!)^{2H_0-2} \Bigg(\int_t^{t+h} \int_{0<s_1<s_2<\cdots<s_n} \sum_{\alpha\in\mathcal A_n}\left(\frac{2}{\kappa}\right)^{\frac{n}{2H_0}} \prod_{j=1}^n  (s_{j+1}-s_j)^{\frac1{2H_0}[2-\frac2\kappa-\frac2\kappa   \alpha_j(1-2H)]}}\\
&{\qquad\qquad \qquad  \left(\int_{\R}\frac{\sin^2(\eta)}{\eta^2} |\eta|^{\frac2\kappa\alpha_j(1-2H)+\frac2\kappa-1}d\eta\right)^{\frac1{2H_0}}d{\bf s}\Bigg)^{2H_0}}
\end{align*} 
where by convention $s_{j+1}=t+h$.

{ Note that in the case $n=1$ we have
\[
B_n(t,h)\leq \frac{2}{\kappa}\left(\int_{\R}\frac{\sin^2(\eta)}{\eta^2} |\eta|^{\frac2\kappa(1-2H)+\frac2\kappa-1}d\eta\right)\frac{h^{2-\frac2\kappa-\frac2\kappa (1-2H)+2H_0}}{(\beta+1)^{2H_0}}
\]
with $\beta=\frac1{2H_0}[2-\frac2\kappa-\frac2\kappa (1-2H)]$,
and in the case $n\geq 2$, it is easy to see from \eqref{beta-n-1} that 
$\beta_1+\dots+\beta_{n-1}+n-1>0$ and hence by Lemma \ref{lem2}, we have }
\begin{align*}
&\int_t^{t+h} \int_{0<s_1<s_2<\cdots<s_n}  \prod_{j=1}^n  (s_{j+1}-s_j)^{\beta_j}   d{\bf s}\\=&\frac{\prod_{j=1}^{n-1}\Gamma(1+\beta_j)}{\Gamma(\beta_1+\cdots+\beta_{n-1}+n)} \int_{t}^{t+h} s_n^{\beta_1+\cdots+\beta_{n-1}+n-1}(t+h-s_n)^{\beta_n}ds_n\\
\le &T^{\beta_1+\cdots+\beta_{n-1}+n-1} \frac{\prod_{j=1}^{n-1}\Gamma(1+\beta_j)}{\Gamma(\beta_1+\cdots+\beta_{n-1}+n)} \frac1{\beta_n+1}h^{\beta_n+1}\\
=&\ T^{\beta_1+\cdots+\beta_{n-1}+n-1} \frac{(\beta_1+\cdots+\beta_{n-1}+n)\prod_{j=1}^{n-1}\Gamma(1+\beta_j)}{\Gamma(\beta_1+\cdots+\beta_{n-1}+n+1)} \frac1{\beta_n+1}h^{\beta_n+1},
\end{align*}
where  $\beta_j=\frac1{2H_0}[2-\frac2\kappa-\frac2\kappa   \alpha_j(1-2H)]$.  { Now similar to the calculus below the equation \eqref{beta-n-1}} and recalling that  $\alpha_n\in\{0,1\}$, we can show that 
\begin{equation}\label{eq4.3}
\sum_{n\ge 0}(p-1)^{\frac n2}(n!)^{\frac12}  \sqrt{B_n(t,h)}\le C (|h|^{1-\frac1\kappa +H_0}+|h|^{1-\frac2\kappa+\frac{2H}{\kappa}+H_0})
\end{equation}
 with $C$ depending on $ (p, H_0, H, T)$.
 
 Finally, combining inequalities \eqref{eq4.1}, \eqref{eq4.2} and \eqref{eq4.3}, for $|h|\leq 1$, we have 
 \[\|u(t+h, x)-u(t, x)\|_p\le C |h|^{\theta_1} \]
 for any $\theta_1\in(0, 1-\frac2\kappa+\frac{2H}\kappa)$ where $C$ is a constant depending only on $(p,\kappa, H_0, H, T, \mathbb M, \theta_1),$  and the H\"older continuity in space follows from the Kolmogorov's continuity criterion. 
 
 The proof is concluded. \hfill
\end{proof}

\section{Appendix}\label{sec-appendix}
In this section, we collect the lemmas that were used in the preceding sections. Some of the proofs are obvious and hence omitted. 

\begin{lemma}\label{lem1}
For $a>0$ and $\theta>-1$, 
$$\int_{\R}\exp(-ax^2)|x|^\theta dx=a^{-\frac12(1+\theta)} \int_\R \exp(-x^2)|x|^\theta dx.$$
\end{lemma}
\begin{lemma}\label{lem2}
Suppose $\alpha_i\in (-1, \infty), i=1, \dots, n$ and let $\alpha=\alpha_1+\dots+\alpha_n$.
Then \[\int_{[0< r_1<\dots<r_n<{r_{n+1}=t}]}~ \prod_{i=1}^{n}({ r_{i+1}-r_{i}})^{\alpha_i}~dr_1\dots dr_n= \frac{\prod_{i=1}^n\Gamma(\alpha_i+1)t^{\alpha+n}}{\Gamma(\alpha+n+1)},\]
where $\Gamma(x)=\int_0^\infty t^{x-1} e^{-t}dt$ is the Gamma function.
\end{lemma}

\begin{lemma}\label{lem22}
For any $a>0$ and $b\in[0,1]$, we have 
\begin{equation}\label{gamma-appr}
\lim\limits_{n\to\infty}\frac{\Gamma(an+b)}{(n!)^aa^{an+b-\frac12}n^{b-\frac12-\frac a2}}=1,
\end{equation}
and 
\begin{equation}\label{power-est}
c_1exp\left(c_2 x^{\frac{1}{a}}\right)\leq \sum_{n=0}^\infty \frac{x^n}{(n!)^a}\leq C_1exp\left(C_2 x^{\frac{1}{a}}\right), \quad \forall x>0,
\end{equation}
where $c_1>0$, $c_2>0$, $C_1>0$ and $C_2>0$ are some constants depending on $a$.
\end{lemma}

\begin{proof}
The proof of \eqref{gamma-appr} follows from Stirling's formula (see also (68) in \cite{bc16} which is \eqref{gamma-appr} in the case of $b=1$). See Lemma A. 1 in \cite{bc16} for the upper bound in \eqref{power-est} and  Lemma 5.2 in \cite{bjq17} for the lower bound in \eqref{power-est}.\hfill
\end{proof}

\begin{lemma}\label{lem5}

\[\int_0^\infty \sin^2(x) x^{-\alpha}dx<\infty\]
if and only if $\alpha\in(1,3).$
\end{lemma}
\begin{proof}
The sufficiency is obvious. The necessity follows from the estimation
\begin{align*}
\int_0^\infty \sin^2(x) x^{-\alpha}dx&\geq \int_0^{\frac{\pi}{4}}\frac{\sin^2(x)}{x^2}\,x^{2-\alpha}dx+\sum_{n=0}^\infty \int_{(n+1/4)\pi}^{(n+3/4)\pi} \sin^2(x) x^{-\alpha}dx\\
&\ge\int_0^{\frac{\pi}{4}}\frac{\sin^2(x)}{x^2}\,x^{2-\alpha}dx+ \frac1{4}\pi^{1-\alpha}\sum_{n=0}^{\infty}(n+3/4)^{-\alpha}.
\end{align*} 
The proof is completed.\hfill
\end{proof}

  \begin{lemma} \label{lem6} 
  For $H\in(0, 1)$ and ${r}, s>0$, 
  \[ C_H\int_{\R} \sin(r|\eta|)\sin(s|\eta|)|\eta|^{-1-2H}d\eta =\frac14\big(|r+s|^{2H}-|r-s|^{2H}\big).\]
  with $C_H$ given by \eqref{eq-CH}. In particular, this integral is positive. 
  \end{lemma}

\begin{proof} 
Let $\mathbb X^H$ be the Hilbert space associated with fractional Brownian motion $\{B^H(x), x\in\R\}$ with Hurst parameter $H\in(0,1)$, i.e., it is the linear expansion of indicator functions under the inner product \[\langle I_{[0,x]}, I_{[0,y]} \rangle_{\mathbb X^H} =\frac12(|x|^{2H}+|y|^{2H}-|x-y|^{2H}).\]
Using the convention $I_{[0,x]}=-I_{[x,0]}$ for $x<0$, the linear mapping $B^H: I_{[0,x]}\to B^H(x)$ extends to a linear isometry between $\mathbb X^H$ and the Gaussian space $\{B^H(\varphi), \varphi\in \mathbb X^H\}$ spanned by $B^H$. Furthermore, for $\varphi, \phi\in\mathbb X^H$, we have (see \cite{jolis})
\[
\langle \varphi, \phi\rangle_{\mathbb X^H}=C_H\int_\R \wh\varphi(\eta)\overline{\wh\phi(\eta)} |\eta|^{1-2H}d\eta.
 \]
Now, noting that $(\mathcal F I_{[-r, r]}(\cdot))(\xi)=\overline{(\mathcal F I_{[-r, r]}(\cdot))(\xi)}=\frac{2\sin(r|\xi|)}{|\xi|}$, we have
\begin{align*}
& C_H \int_{\R} \sin(r|\eta|)\sin(s|\eta|)|\eta|^{-1-2H}d\eta\\
=& \frac14C_H \int_\R (\mathcal F I_{[|x|\le r]})(\eta)(\mathcal F I_{[|x|\le s]})(\eta)  |\eta|^{1-2H} d\eta\\
=& \frac14 \langle I_{[-r, r]}(\cdot), I_{[-s, s]}(\cdot)\rangle_{\mathbb X^H}\\
=&\frac14 \E[B^H(I_{[-r, r]}(\cdot))B^H(I_{[-s, s]}(\cdot))]\\
=&\frac14 \E[(B^H(r)-B^{H}(-r))(B^H(s)-B^{H}(-s))]\\
=&\frac14\big(|r+s|^{2H}-|r-s|^{2H}\big).
\end{align*}
We complete the proof.\hfill 
\end{proof}

\begin{lemma}\label{lem8}
For any $a\geq0$, we have
\[\int_0^\infty e^{-t}\sin(at)dt=\frac{a}{1+a^2}\]
\end{lemma}
\begin{proof} The result follows by integration by parts.
\hfill
\end{proof}

\begin{lemma}\label{lem3}
Let $a$ and $b$ be two positive constants,  and $\lambda\in[1,\infty), \beta\in[0,1), \gamma\in[0,1]$ such that $1-\lambda<\beta+2\gamma<1$. Then we have 
\[\int_\R \sin^2(a x) |x|^{-2+\beta} ((b |x|^\lambda) \wedge 2) dx \le C a^{2\gamma}b^{\frac1\lambda(1-\beta-2\gamma)}, \]
where \[C=C_{\lambda, \beta,\gamma}=\max\left\{\int_{|y|\le2}{|y|}^{\frac1\lambda(-1+\beta+2\gamma)}{dy}, ~~\int_{|y|>2}{|y|}^{\frac1\lambda(-1+\beta+2\gamma)-1}{dy}\right\}.\]
 \end{lemma}
\begin{proof} We write
\begin{align*}
&\int_\R \sin^2(a x) |x|^{-2+\beta} ((b |x|^\lambda)\wedge 2) dx\\
=& \int_{b|x|^\lambda \le2} \sin^2(a x) |x|^{-2+\beta} b |x|^\lambda dx+2\int_{b|x|^\lambda>2} \sin^2(a x) |x|^{-2+\beta} dx.
\end{align*}
Noting that ${|\sin(x)|}\le |x|^{\gamma}$ for $\gamma\in[0,1]$, we have
\begin{align*}
\int_{b|x|^\lambda\le2} \sin^2(a x) |x|^{-2+\beta} b |x|^\lambda dx
&\le \int_{b|x|^\lambda\le2} |ax|^{2\gamma} |x|^{-2+\beta} b |x|^\lambda dx\\
&=\frac1\lambda  a^{2\gamma}b^{\frac{1}\lambda(1-\beta-2\gamma)}\int_{|y|\le2} {|y|^{\frac1\lambda(-1+\beta+2\gamma)}}dy,
\end{align*}
and
\begin{align*}
\int_{b|x|^\lambda>2} \sin^2(a x) |x|^{-2+\beta} dx\le  \int_{b|x|^\lambda> 2} |ax|^{2\gamma} |x|^{-2+\beta}  dx =\frac1\lambda a^{2\gamma} b^{\frac1\lambda(1-\beta-2\gamma)}\int_{|y|>2} |y|^{\frac1\lambda(-1+\beta+2\gamma)-1}dy.
\end{align*}
Thus, the proof is concluded. \hfill
\end{proof}

\begin{lemma}\label{lem66}
For any $t,h\in \R$ and $\gamma\in(0,1]$, there exists a constant $C_\gamma$ such that
\[ |\sin((t+h)x)-\sin(tx)|\le C_\gamma|hx|^\gamma.\]
In particular, 
\[ |\sin((t+h)x)-\sin(tx)|\le C_\gamma(|hx|^\gamma\wedge |hx|).\]

\end{lemma}
\begin{proof}
By the mean value theorem and the fact of $y\le \frac{2^{1-\gamma}}{\gamma}y^\gamma$ for $y\in[0,2]$ and $\gamma\in(0,1]$, we have
\[
|\sin((t+h)x)-\sin(tx)|\le |hx||\cos(sx)|\leq |hx|,
\]
and
\begin{align*}
|\sin((t+h)x)-\sin(tx)|\le \frac{2^{1-\gamma}}{\gamma} |\sin((t+h)x)-\sin(tx)|^\gamma
=C_\gamma|hx|^\gamma |\cos(sx)|^\gamma\leq C_\gamma|hx|^\gamma,
\end{align*}
where $s$ is a number between $t$ and $t+h$. The desired results can be obtained.
\hfill
\end{proof}

\section*{Acknowledgements}
F. Xu is partially supported by National Natural Science Foundation of China (Grant No.11871219, No.11871220) and 111 Project (B14019). 

$\begin{array}{cc}
\begin{minipage}[t]{1\textwidth}
{\bf Jian Song}\\
School of Mathematics, Shandong University, Jinan, Shandong 250100, China\\
\texttt{txjsong@hotmail.com}
\end{minipage}
\hfill
\end{array}$
\medskip

$\begin{array}{cc}
\begin{minipage}[t]{1\textwidth}
{\bf  Xiaoming Song}\\
Department of Mathematics, Drexel University, Philadelphia, PA 19104, USA\\
\texttt{xs73@drexel.edu}
\end{minipage}
\hfill
\end{array}$

\medskip

$\begin{array}{cc}
\begin{minipage}[t]{1\textwidth}
{\bf Fangjun Xu}\\
Key Laboratory of Advanced Theory and Application in Statistics and Data Science - MOE, School of Statistics, East China Normal University, Shanghai 200062, China\\
NYU-ECNU Institute of Mathematical Sciences, NYU Shanghai, Shanghai 200062, China\\
\texttt{fangjunxu@gmail.com, fjxu@finance.ecnu.edu.cn}
\end{minipage}
\hfill
\end{array}$

\end{document}